\documentclass[11pt]{article}
\usepackage{amssymb,amsmath,amsthm}
\usepackage{epsfig}
\usepackage{breqn}
\usepackage{rotating}
\usepackage{amsfonts}
\usepackage{setspace}
\usepackage{fullpage}
 \usepackage{bbold} 
\usepackage{amsmath}
\usepackage{comment}
\usepackage{pgf,tikz}
\usepackage{mathtools}  
\usepackage{hyperref}

\usepackage{comment}
\usepackage{authblk}

\usepackage{setspace}
\usepackage{bbm}
\newtheorem{theorem}{Theorem}
\newtheorem{claim}[theorem]{Claim}

\newtheorem{remark}{Remark}

\newcommand{\ceil}[1]{\left\lceil{#1}\right\rceil}

\def\BP{\mathcal {BP}}

\def\C{\mathcal C}

\def\h{\mathcal H}

\newenvironment{claimproof}[1][\myproofname]{\begin{proof}[#1]}{\end{proof}}
\newcommand*{\myproofname}{Proof}

\newcommand{\abs}[1]{\left\lvert{#1}\right\rvert}

\author[1]{Ervin Győri}
\author[2]{Nika Salia}

\affil[1]{Alfr\'ed R\'enyi Institute of Mathematics.}
\affil[2]{Extremal Combinatorics and Probability Group, Institute for Basic Science, Daejeon, South Korea.}

\title{Linear three-uniform hypergraphs with no Berge path of given length}

\begin{document}
\maketitle
\begin{abstract}
Extensions of Erdős–Gallai Theorem for general hypergraphs are well studied. 
In this work, we prove the extension of Erdős–Gallai Theorem for linear hypergraphs.
In particular, we show that the number of hyperedges in an $n$-vertex $3$-uniform linear hypergraph,  without a Berge path of length $k$ as a subgraph  is at most $\frac{(k-1)}{6}n$ for  $k\geq 4$. 
\end{abstract}

\section{Introduction}

Finding the maximum number of edges in a graph with fixed order not containing another graph as a subgraph is a central problem in extremal combinatorics. This work considers problems where a path of fixed length is forbidden.  This problem is well understood for graphs and $r$-uniform hypergraphs. 
Erdős–Gallai theorem states that a graph of order $n$ containing no path of length~$k$ as a subgraph contains at most $\frac{k-1}{2}n$ edges. 
This bound is sharp for infinitely many $n$.  
In particular, equality holds if and only if $n$ is a multiple of $k$ and the graph is isomorphic to the union of $\frac{n}{k}$ cliques of size $k$. 
This theorem was extended to $r$-uniform hypergraphs by Gy{\H{o}}ri, Katona and Lemons~\cite{gyHori2016hypergraph}. In order to state their result, we will introduce the necessary definitions.

For an integer $r$, a hypergraph $\h$ is $r$-uniform if it is a family of $r$-element sets of finite family $V(\h)$. 
We will use the following extension of this definition. For a set of integers  $R$, a hypergraph $\h$ is $R$-uniform if it is a family of sets of finite family $V(\h)$, such that the sizes of the sets are elements of $R$. 
Paths in hypergraphs can be defined in number of ways. In this paper we fallow definition of Berge~\cite{berge1973graphs}. A Berge path of length $k$ in a hypergraph $\mathcal H$ is an alternating sequence  $v_1,h_1,v_2,\dots,h_k,v_{k+1}$ of distinct vertices and hyperedges such that $\{v_{i},v_{i+1}\}\subseteq  h_i$ for all $i\in[k]$. A Berge cycle of length $k$ is also defined similarly. The vertices $v_i$, $i\in[k+1]$, are defining vertices of the Berge path and the hyperedges $h_i$, $i\in[k]$, are defining hyperedges of the Berge path.

\begin{theorem}[Gy{\H{o}}ri, Katona and Lemons~\cite{gyHori2016hypergraph}]
Let $\h$ be an $n$-vertex $r$-uniform hypergraph containing no Berge path of length $k$ as a subgraph.
Then if $r\geq k>2$ then the number of hyperedges of $\h$ is at most $\frac{k-1}{r+1}n$.
If $k > r+1 >2$ then the number of hyperedges of $\h$ is at most  $ \frac{\binom{k}{r}}{k}n$.
\end{theorem}
The remaining case $k = r + 1$ was settled later in~\cite{davoodi2018erdHos}. Forbidden path problems for connected graphs and hypergraphs including their stability versions are well studied, we refer interested readers to~\cite{kopylov1977maximal, balister2008connected,gyHori2018maximum,furedi20192,gyHori2019connected, furedi2016stability, furedi2018stability, gerbner2020stability}. $r$-uniform hypergraphs with bounded circumference was studied in~\cite{furedi2019avoiding, gyHori2020structure} and references therein.   

Here we introduce some necessary technical definitions. 
For a hypergraph $\h$ let $E(\h)$ be the hyperedge set and $V(\h)$ be the vertex set, we denote their sizes by $e(\h)$ and $v(\h)$ accordingly. For  a vertex set $V$, $V\subseteq V(\h)$, we define another hypergraph $\h_{V}$. Where $V(\h_{V})=V$ and $E(\h_{V})=\{h\setminus V: h\in E(\h),\abs{h\setminus V}\geq 2\}$. Note that if $\h$ is $\{2,3\}$-uniform linear hypergraph then $\h_V$ is $\{2,3\}$-uniform linear hypergraph also. 
The induced hypergraph on the vertex set $V$ is denoted by $\h[V]$.
For a hypergraph $\h$ we denote two-shadow of $\h$ by $\partial \h$.  $\partial \h$ is  a graph on the same vertex set as $\h$ and the set of edges is $\{\{u,v\}:\{u,v\}\subseteq h \in E(\h)\}$. The degree of a vertex $v$ in a hypergraph $\h$ is the number of hyperedges incident to the vertex $v$ and is denoted by $d_{\h}(v)$. The minimum degree of a vertex in a hypergraph $\h$ is denoted by $\delta_{\h}(v)$. The circumference of $\h$ is the length of the longest Berge cycle in a hypergraph $\h$ and is denoted by $c(\h)$.  The neighborhood of a vertex $v$ in a hypergraph $\h$ is denoted by $N_{\h}(v)$. For a hypergraph $\h$ and sub-hypergraph $\h'$ we denote the hypergraph on the same vertex set as $\h$ and hyperedge set $E(\h)\setminus E(\h')$ by $\h \setminus \h'$.

\section{Main results}

In this paper, we prove the extension of Erd{\H{o}}s-{G}allai theorem for linear $3$-uniform hypergraphs.

\begin{theorem}\label{thm:uniform}
Let $\h$ be an $n$ vertex $3$-uniform linear hypergraph, containing no Berge path of length $k\geq 4$. Then the number of hyperedges in $\h$ is at most $\frac{k-1}{6}n$.
\end{theorem}

Note that the upper bound is sharp for infinitely many $k$ and $n$. In particular for all $k$ for which there exists Steiner Triple System (a 3-uniform hypergraph that every pair of vertices is covered by exactlx one hyperedge) and $n$ multiple of $k$ there exists an $n$-vertex $3$-uniform linear hypergraph $\h$, containing no Berge path of length $k$ with  $\frac{k-1}{6}n$ hyperedges. 
Where $\h$ is the disjoint union of $\frac{n}{k}$ copies of $k$-vertex Steiner Triple Systems.

In order to prove Theorem~\ref{thm:uniform} with induction for $k$,  we need a stronger and more general statement of the theorem.  
\begin{theorem}\label{thm:main}
Let $\h$ be an $n$ vertex $\{2,3\}$-uniform linear hypergraph, containing no Berge path of length $k\geq 4$. Then the number of edges in $\partial \h$ is at most $\frac{k-1}{2}n$.
\end{theorem}
 
 It is easy to settle the remaining cases when $k\leq 4$. So Theorem~\ref{thm:uniform} is a direct corollary of Theorem~\ref{thm:main}. 
\begin{remark}\label{Remark:k=1,2,3Classification}
Let $\h$ be an $n$ vertex linear $\{2,3\}$-uniform hypergraph, containing no Berge path of length $k$.  
\begin{itemize}
    \item  If $k=1$ then $e(\partial \h)=0$;
    \item If $k=2$  then $e(\partial \h)\leq v(\h)$; The upper-bound is sharp and the equality is achieved if and only if is $v(\h)$ multiple of $3$ and $\h$ is  $\frac{v(\h)}{3}$ independent hyperedges of size three.
    
    \item If $k=3$ then $e(\partial \h)\leq 3\frac{v(\h)-1}{2}$. The upper-bound is sharp and the equality is achieved if and only if $v(\h)$ is odd and $\h$ is  $\frac{v(\h)-1}{2}$  hyperedges of size three sharing the same vertex for every $n\geq 3$.
\end{itemize}
\end{remark}

\section{Proof of Theorem~\ref{thm:main} }

We prove Theorem~\ref{thm:main} by induction on $k$.  At first, we consider the base case when $k=4$. We may assume $\h$ is a connected hypergraph since the upper bound is linear for $n$ and the additive constant is 0. 
If $\h$ is Berge cycle free then $e(\partial \h)\leq \frac{3(n-1)}{2}$  (the upper-bound is attained by hyperedges of size three sharing a fixed vertex). 
If $\h$ contains a Berge cycle it must be a Berge cycle of length $3$ or~$4$ since it is a linear hypergraph. 
If $\h$ contains Berge cycle of length $4$ then by connectivity $v(\h)\leq 4$, hence $e(\h)\leq \binom{4}{2}= \frac{3n}{2}$.
If $\h$ contains a cycle of length $3$, we denote it by $C_3$. Cycle $C_3$ is a linear cycle since $\h$ is a linear hypergraph. If all of the hyperedges of $C_3$ are size three then by the connectivity of $\h$ we have  $\h=C_3$ and $e(\partial \h)=9= \frac{3n}{2}$. 
If two of the hyperedges are size three then by the connectivity of $\h$ we have 
 $\h=C_3$ and  $e(\partial \h)=7< \frac{3n}{2}$. 
 If at most one hyperedge is size three then we have  $e(\partial \h)\leq \frac{3n}{2}$.  So the base case $k=4$ is done. 

Let $\h$ be an $n$-vertex linear $\{2,3\}$-uniform hypergraph containing no Berge path of length $k$ for some integer $k>4$. 
Suppose by way of contradiction that  $e(\partial \h)> \frac{n(k-1)}{2}$.
Without loss of generality, we may assume $n$ is minimal, in particular, we assume all linear $\{2,3\}$-uniform hypergraphs containing no Berge path of length $k$ with $n'$ vertices, $n'<n$, contain at most $\frac{n(k-1)}{2}$ edges in the shadow. 
Note that from the minimality of $n$ we have  the hypergraph $\h$ is connected.
Even more, for each vertex $v$, $\h_{V(\h)\setminus \{v\}}$ contains no Berge path of length $k$, thus from the minimality   of $n$ we have $d_{\partial\h}(v)>\frac{k-1}{2}$. Hence we have $\delta_{\partial\h}(v)\geq \ceil{\frac{k}{2}}$.   Note that since $e(\partial \h)> \frac{n(k-1)}{2}$ the longest path of $\h$ is length $k-1$ by the induction hypothesis.

\begin{claim}\label{Claim:Minimum_Circumferance}
    $c(\h)\geq \ceil{\frac{k+1}{2}}$
\end{claim}
\begin{claimproof}
Let $P:=v_1,h_1,v_2,h_2,\dots,h_{k-1},v_{k}$ be a longest Berge path of $\h$.
Let $u$ be a vertex adjacent to vertex $v_1$ with a hyperedge $h_u$. 
Then either $u\in \{v_2,v_3,\dots,v_{k}\}$ or $h=h_1=\{v_1,v_2,u\}$ since $P$ is a longest Berge path.  
Therefore all vertices from $N_{\partial\h}(v_1)$ and $N_{\partial\h}(v_k)$ are defining vertices of the Berge path but at most one.  
Evenmore, if $v_1$ is adjacent to vertex $v_i$ with a non-defining hyperedge $h_{v_i}$ then we have a Berge cycle of length $i$, namely $v_1,h_1,v_2,h_2,\dots,v_i,h_{v_i},v_1$.
If $h_{v_i}$ is a defining hyperedge of $P$ then it is either $h_{i-1}$ or $h_i$ since $\h$ is $3$-uniform. 
Therefore  there is a cycle of length $i-1$ or $i$. 
Hence either $c(\h)\geq \ceil{\frac{k+1}{2}}$ or
\[
N(v_1)\subseteq \{v_2,v_3,\dots v_{\ceil {\frac{k+1}{2}}}\}\cup  \left(h_1 \setminus \{v_1,v_2\dots,v_k\} \right).
\]
Note that if $v_1$ is incident to the vertex $v_{\ceil {\frac{k+1}{2}}}$ then either $c(\h)\geq \ceil{\frac{k+1}{2}}$ or $h_{1}=\{v_1,v_{\ceil {\frac{k+1}{2}}},v_{2}\}$ or 
$h_{\ceil {\frac{k-1}{2}}}=\{v_1,v_{\ceil {\frac{k-1}{2}}},v_{\ceil {\frac{k+1}{2}}}\}$.

If $k$ is odd the degree of $v_1$ is at least $\frac{k+1}{2}$ in $\partial \h$. Thus $N_{\partial \h}(v_1)= \{u_1,v_2,v_3,\dots v_{\frac{k+1}{2}}\}$, where $h_1=\{u_1,v_1,v_2\}$ and $u_1\notin \{v_i:i\in[k]\}$.  Note that the vertex $v_1$ can be exchanged with the vertex $u_1$, thus $N_{\partial \h}(u_1)= \{v_1,v_2,v_3,\dots v_{\ceil {\frac{k+1}{2}}}\}$ or  $c(\h)\geq \ceil{\frac{k+1}{2}}$ and we are done. 
The hyperedge $h_3$ is not incident to either $v_1$ or $u_1$, without loss of generality we may assume $u_1\notin h_3$. Since $u_1$ and $v_3$ are adjacent vertices, let $h'$ be a hyperedge incident to both. Note that if $h'$ is a $P$ defining hyperedge it must be $h_2$. The following is a longest Berge path, $v_2, h_1,u_1,h',v_3,h_3,\dots,v_k$. Thus similarly as for $v_1$ we have $N_{\partial \h}(v_2)= \{u_1,v_1,v_3,\dots v_{\frac{k+1}{2}}\}$ or  $c(\h)\geq \ceil{\frac{k+1}{2}}$ and we are done. 
The number of edges incident to $v_1,v_2,u_1$ in $\partial \h$ is $3\frac{k+1}{2}-3$, a contradiction to the minimality of $\h$ since the hypergraph $\h_{V(\h)\setminus \{v_1,v_2,u_1\}}$ contains more that $\frac{(n-3)(k-1)}{2}$ edges in the shadow and no Berge path of length $k$.

If $k$ is even and if $v_1$ is incident to $v_{\frac{k+2}{2}}$ then  $N_{\partial \h}(v_k)= \{u_k,v_{\frac{k+2}{2}},v_{\frac{k+4}{2}},\dots, v_{k-1}\}$, where $h_{k-1}=\{u_k,v_k,v_{k-1}\}$ and $u_k\notin \{v_i:i\in[k]\}$, or  $c(\h)\geq \ceil{\frac{k+1}{2}}$ and we are done. In this case we may finish the proof with the same argument provided for   odd $k$.

If $k$ is even and if $v_1$ is not incident to $v_{\frac{k+2}{2}}$ then $N_{\partial \h}(v_1)= \{u_1,v_2,v_3,\dots v_{\ceil {\frac{k}{2}}}\}$, where $h_1=\{u_1,v_1,v_2\}$ and $u_1\notin \{v_i:i\in[k]\}$, or  $c(\h)\geq \ceil{\frac{k+1}{2}}$ and we are done.   In this case, we may finish the proof with the same argument provided for odd~$k$.
\end{claimproof}

Let $\C_{\ell}:=v_1,h_1,v_2,h_2,\dots h_{\ell-1},v_{\ell},h_{\ell},v_1$ be a longest Berge cycle of $\h$. 
Some $\C_{\ell}$  defining hyperedges $h_i$ are size three, let us denote the third vertex by $x_i$, that is $h_i=\{v_i,v_{i+1},x_i\}$ for hyperedges of size three.
From Claim~\ref{Claim:Minimum_Circumferance} we have $\ell\geq \ceil{\frac{k+1}{2}}$.
Let us denote the hypergraph $\h_{V(\h)\setminus \{v_i:i\in[\ell]\}}$ by $\h'$. 

\begin{claim}\label{Claim:h'_pathlength}
The hypergraph $\h'$ is $\BP_{k-\ell}$-free.   
\end{claim}
\begin{claimproof}
    Consider a longest Berge path $P'$ in $\h'$, and let $P:=w_1,e_1,w_2,\dots,e_{k'},w_{k'+1}$ be the corresponding path on the same defining vertex set as $P'$ in $\h$. Since $\h$ is $3$-uniform the defining hyperedges of $P$ and $C_{\ell}$  are distinct. 
    If a terminal vertex  of $P$ is adjacent to a defining vertex of $C_{\ell}$ with a hyperedge distinct from $e_1$ then $\h$ contains a Berge path of length $\ell+k'$, thus $k'< k-\ell$ and we are done. 

If $w_1$ is adjacent to a defining vertex of $C_{\ell}$ with hyperedge $e_1$, then $k'\leq k-\ell$. 
The degree of $w_1$ in $\partial \h$ is at least $\ceil{\frac{k}{2}}$, and all the neighbors of $w_1$ are defining vertices of $P$ but one. Thus $k'\geq \ceil {\frac{k}{2}}-1$ and  either we have a contradiction by finding a path of length $k$ in $\h$ or $\ell=\ceil{\frac{k+1}{2}}$ and $k'=\ceil{\frac{k-2}{2}}$.
By the minimum degree condition, the vertex $w_k$ is incident to a $P$ non-defining vertex with the hyperedge $e_k$. 
Even more the vertex $w_1$ is incident to $w_{k'}$ with a hyperedge which is not $P$ defining. Thus $\h'$ contains a Berge cycle $C'$ of length  $\ceil{\frac{k}{2}}$.  By the connectivity of $\h$ there is a Berge path containing all defining vertices of both cycles $C_{\ell}$ and  $C'$, thus the length of the Berge path is at least $k$ a contradiction.

If neither $w_1$ nor $w_{k'+1}$ are adjacent to a vertex of $C_{\ell}$, then by the same argument as in Claim~\ref{Claim:Minimum_Circumferance} we have $c(\h')\geq \ceil{\frac{k+1}{2}}$ or a path of length at least $k$ in $\h$ which is a contradiction. 
Since $\h$ is connected there is a path in $\h$ containing all vertices of the cycle $C_{\ell}$ together with all vertices of a longest cycle in $\h'$. 
A contradiction since both cycles contains at least $\ceil{\frac{k+1}{2}}$ vertices.
\end{claimproof}

If $k-\ell \geq 4$ then by Claim~\ref{Claim:h'_pathlength} and  induction hypothesis  for hypergraph $\h'$ we have 
\begin{equation}\label{Equation:the_number_of_edges_in_dH'}
e(\partial \h')\leq \frac{(n-\ell)(k-\ell-1)}{2}.
\end{equation}
For a vertex $u \in V(\h')$ we define the set 
$S(u):=N_{\h\setminus C_{\ell}}(u)\cap V(C_{\ell})$, $L(u):=\{v_i:u=x_i\}$ and   $R(u):=\{v_{i+1}:u=x_{i}\}$.  
For a vertex set $S$ such that  $S\subseteq V(C_{\ell})$ let $S^+$ be a set $S$ shifted right, in particular 
$S^+:=\{v_i:v_{i-1}\in S\}$, the indices are taken module $\ell$. 
Similarly we definite $S^-$, in particular $S^-$ is a set for which $S=(S^-)^+$. 
Naturally we denote the set $(S^-)^-$ with $S^{--}$ and the set $(S^+)^+$ with $S^{++}$. 
Note that $L(u)^+=R(u)$, thus the size of $L(u)$ and $R(u)$ are the same.

In what follows we are going to estimate the number of edges in $\partial \h$, in the following way
\begin{equation}\label{Equation:The_number_of_edges}
    e(\partial \h)=e(\partial \h_{V(C_{\ell})})+e_{\partial \h }(V(C_{\ell}),V(\h'))+e(\partial \h'). 
\end{equation}
Noting that $e_G(A,B)$ denotes the number of edges between vertex set $A$ and $B$ in $G$.
In most cases, we will use a naive upper bound for $e(\partial \h_{V(C_{\ell})})\leq \binom{\ell}{2}$. 
For $k-\ell \geq 4$, we estimate $e(\partial \h')$ by the induction hypotheses as in Equation~\ref{Equation:the_number_of_edges_in_dH'}. 
We estimate the number of edges from $V(\h')$ to $V(C_{\ell})$, for each vertex $u\in V(\h')$ in $\partial \h$.  In particular the number of adjacent vertices to $u$ is $\abs{L(u)}+\abs{R(u)}+\abs{S(u)}$. 
Since each defining hyperedge of $C_{\ell}$ provides at most two edges crossing between the vertices $V(\h')$ and $V(C_{\ell})$ we have a naive  upper bound for $e_{\partial \h }(V(C_{\ell}),V(\h'))$ which is enough for most of the cases. 
\begin{equation}\label{Equation:crossing_number}
    e_{\partial \h }(V(C_{\ell}),V(\h'))\leq 2\ell+\sum_{u\in V(\h')} \abs{S(u)}.
\end{equation}

Since $C_{\ell}$ is a longest Berge cycle of $\h$ we are able to get an upper bound for $\abs{S(u)}$ from the following claim. 
\begin{claim}\label{claim:+}
 For a vertex $u \in V(\h')$ we have $(S(u)\cup L(u))\cap S(u)^-=\emptyset$. 
\end{claim}

\begin{proof}
  If there is  $v_i$ such that $v_i \in (S(u)\cup L(u))\cap S(u)^-$ then there are  hyperedges $h'$ and $h''$ such that $h'$ is incident to $v_{i}, u$ and hyperedge $h''$ is incident to $v_{i+1}, u$, where hyperedge $h''$ is not a defining hyperedge of $C_{\ell}$. 
  Since $\h$ is a linear hypergraph $h'$ and $h''$ are distinct hyperedges, even more, if $h'$ is a  defining hyperedge of $C_{\ell}$ then $h'=h_i$. 
  We may extend $C_{\ell}$ with the vertex $u$, in particular  instead of $v_{i},h_{i},v_{i+1}$ we take $v_{i},h',u,h'',v_{i+1}$ in the cycle $C_{\ell}$, a contradiction. 
\end{proof}

Note that if a vertex $v_i\in S(u)$ then $v_{i+1}\notin S(u)$ from Claim~\ref{claim:+}. Thus we have $\abs{S(u)}\leq \frac{\ell}{2}$ for each vertex $u$ of $\h'$. 
Therefore $e_{\partial \h }(V(C_{\ell}),V(\h'))\leq 2 \ell +\frac{\ell}{2}(n-\ell)$ from Equation~\ref{Equation:crossing_number}.  
If $k-\ell\geq 4$ then by Equation~\ref{Equation:The_number_of_edges} and ~\ref{Equation:the_number_of_edges_in_dH'} we have a contradiction
\[
e(\partial \h) \leq \binom{\ell}{2}+2\ell+ \frac{\ell(n-\ell)}{2}+ \frac{(n-\ell)(k-\ell-1)}{2}=\frac{n(k-1)}{2}+ \frac{\ell}{2}(\ell+4-k)\leq  \frac{n(k-1)}{2}.
\]
We study the rest of the possible values of $\ell$ separately,  $\ell \in \{k-3,k-2,k-1,k\}$. Let $x$ be the number of defining hyperedges of $C_{\ell}$  incident to a vertex of $\h'$. Note that $0\leq x\leq \ell$.

If $\ell=k$  then $\C_{\ell}=\h$ otherwise we have a Berge path of length $k$  in $\h$ by the connectivity of~$\h$. Thus we have  $n=k=\ell$ and
\[
e(\partial \h) \leq \binom{\ell}{2}= \frac{n(k-1)}{2}.
\]

If $\ell=k-1$ then $\h'$ contains no hyperedge by Claim~\ref{Claim:h'_pathlength}. 
Since $\h$ does not contain a Berge path of length $k$, if a hyperedge $h_i$ adjacent to a vertex from $V(\h')$, then neither $v_i$ nor $v_{i+1}$ is a vertex of $S(u)$, for all $u\in V(\h)$. In particular for $u,u'\in V(\h')$ we have $L(u)\cap (S(u))^-=\emptyset$. By this observation and  Claim~\ref{claim:+} every vertex of $V(\h')$ is adjacent to at most $\frac{k-1-x}{2}$ vertices of $C_{\ell}$ with a non-defining hyperedge, that is $\abs{S(u)}\leq \frac{k-1-x}{2}$. Thus by Equation~\ref{Equation:The_number_of_edges} we have 
\[
e(\partial \h) \leq \binom{k-1}{2}+2x+\frac{k-1-x}{2}(n-(k-1)).
\]
Hence if  $n\geq k+2$ or $n=k+1$ and $x\leq  \frac{k-1}{2}$ then we have $e(\partial \h) \leq  \frac{n(k-1)}{2}$, since $x\leq k-1$. As $e(\partial \h )>\frac{n(k-1)}{2}$ we have $n \geq k+1$.

If $n=k+1$ and $x>\frac{k-1}{2}$ then there are two $C_{\ell}$ non-defining hyperedges $h_i$ and $h_{i+1}$ such that $\{x_i,x_{i+1}\}= V(\h')$. 
Since $\h$ does not contain a Berge path of length $k$, if a defining vertex of $C_{\ell}$ is incident to both vertices of $\h'$, either both incidences are from a defining hyperedge or both incidences are from a non-defining hyperedge. 
If $v_j$ is incident to both vertices of $\h'$ with $h_{j-1}$ and $h_j$ such that $j\neq i-1,i,i+1$ then $v_{j}$ is not incident to $v_{i+1}$. Otherwise, if there is a hyperedge $f'$ incident to $v_{j}$ and $v_{j+1}$, then it is a non-defining hyperedge and the following is a Berge path or a Berge cycle of length $k$,
\[
x_{i+1},h_{i+1},v_{i+2},\dots,v_{j},f',v_{i+1},h_{i},v_i,\dots,v_{j+1},h_{j},x_j.
\]
If a vertex $v_j$ is adjacent to $x_1$ or $x_2$ with a non-defining hyperedge then $v_{j+1}$ is not adjacent to a vertex from $\{x_1,x_2\}$.
Thus for each vertex $v_j\in V(C_{\ell})$, $j\notin \{i-1,i,i+1\}$, either there is at most one vertex from $V(\h')$ adjacent to it, or if there are two then $v_jv_{i+1}$ is not an edge of $\partial \h$ or $v_{j+1}$ is not adjacent to any vertex of $V(\h')$. 
Note that if there is a defining hyperedge of  $C_{\ell}$ not incident to  a vertex of $\h'$ then we may choose $i$ such that $i-1$ has exactly one neighbor in $V(\h')$. 
If all defining hyperedges of  $C_{\ell}$ are  incident to  a vertex of $\h'$ then we may choose any $i$ from $[k-1]$.
Thus we have a contradiction from Equation~\ref{Equation:The_number_of_edges}
\[
e(\partial \h) \leq \binom{k-1}{2}+k-1+2\leq \frac{n(k-1)}{2}.
\]

For the next cases, we need another claim, which can be seen as a strengthening of Claim~\ref{claim:+}.

 \begin{claim}\label{claim:+,++}
 Let $\{u_1,u_2,u_3\}$ be a hyperedge of $\h'$. Then for each $i$ and $j$ such that $1\leq i<j\leq 3$ we have $\left(S(u_i)\cup L(u_i)\right)\cap (S(u_j)\cup L(u_j))^-=\emptyset$ and $(S(u_i)\cup L(u_i))\cap S(u_j)^{--}=\emptyset$. 
\end{claim}

\begin{proof}
Suppose by the way of contradiction for some $i$ and $j$ such that $1\leq i<j\leq 3$ we have $v_{\gamma}\in \left(S(u_i)\cup L(u_i)\right)\cap (S(u_j)\cup L(u_j))^-$. 
Thus there is a hyperedge $h'$ incident to vertices $u_i, v_{\gamma}$ and there is a hyperedge $h''$ incident to vertices $u_j, v_{\gamma+1}$. Note that either $h'$ is not $C_{\ell}$ defining hyperedge or it is $h_{\gamma}$, similarly $h''$ is not $C_{\ell}$ defining hyperedge or it is $h_{\gamma+1}$. Thus by exchanging the path $v_{\gamma},h_{\gamma},v_{\gamma+1}$ in $C_{\ell}$ with the path 
\[
v_{\gamma},h',u_i,\{u_1,u_2,u_3\}, u_j, v_{\gamma+1}, 
\]
 we get a cycle of length $\ell+1$, a contradiction. 

In case when $(S(u_i)\cup L(u_i))\cap S(u_j)^{--}\neq \emptyset$  we get a contradiction also in the same way. 


\end{proof}

If $\ell=k-2$ then $\h'$ contains no Berge path of length two, that is it contains disjoint hyperedges of size two and three.
In the following two paragraphs, we show $\h'$ contains no hyperedge of size three, strengthening of Claim~\ref{Claim:h'_pathlength}.

Let $\{u_1,u_2,u_3\}$ be a hyperedge of $\h'$. By Claim~\ref{claim:+,++} if a vertex $v_i$ of $C_{\ell}$ is incident to a vertex from $\{u_1,u_2,u_3\}$  with hyperedge $h_i$ then $v_i$ and $v_{i+1}$ have exactly one neighbour from $\{u_1,u_2,u_3\}$. Even more either $v_{i+2}$ has no neighbour in $\{u_1,u_2,u_3\}$, or it has exactly one neighbour with hyperedge $h_{i+2}$. 
If $v_i$ has at least two neighbours from $\{u_1,u_2,u_3\}$ with a non-defining hyperedge then $v_{i+1}$ is not incident to any of the vertices $\{u_1,u_2,u_3\}$ and if $v_{i+2}$ is adjacent to a vertex from $\{u_1,u_2,u_3\}$ then it is adjacent with hyperedge $h_{i+2}$ by Claim~\ref{claim:+,++}. 
Thus on average every vertex of $C_{\ell}$ has one neighbour from $\{u_1,u_2,u_3\}$. Hence by pigeonhole principle, there is a vertex $u_i$ incident to at most $\frac{k-2}{3}$ vertices from $C_{\ell}$. 
Similarly, it is easy to note that if $k=5$, there is a vertex in $V(\h')$ not adjacent to any of the vertices from $V(C_{\ell})$.
Thus the degree of $u_i$ is at most $\frac{k+4}{3}$ in $\partial \h$. If $k=5$ the degree of $u_i$ is $2$ in $\partial \h$. 
A contradiction to the minimality of $\h$ since $\frac{k+4}{3}< \ceil{\frac{k}{2}}$. 
Thus we have $e(\partial \h')\leq \frac{n-(k-2)}{2}$.

By Claim~\ref{claim:+}, for every vertex $u\in V(\h')$, the number of adjacent vertices from $V(C_{\ell})$ is 
\[
\abs{S(u)}+\abs{L(u)}+\abs{R(u)}\leq \frac{k-2-\abs{L(u)}}{2}+2\abs{L(u)}=\frac{k-2}{2}+1.5\abs{L(u)}.
\]
This bound is not enough for the desired upper bound, in the following paragraph we improve it.

Let $T(u)$ be the number of vertices $v_i\in V(C_{\ell})$ such that $h_{i-1}=\{v_{i-1},v_{i},u\}$,  $v_{i+1}\in S(u)$ and $v_i$ is adjacent to  $v_{i+2}$ with some hyperedge $f_i$. 
Note that for every such vertex $f_i=\{v_i,v_{i+2},v_{i+3}\}$, otherwise the following is a Berge cycle of length $k-1$
\[
v_{i+1},h_i,v_i,f,v_{i+2},h_{i+2},\dots,v_{i-1},h_{i-1},u,h_u,v_{i+1},
\]
where $h_u$ is a $C_{\ell}$ non-defining hyperedge incident to $u$ and $v_{i+1}$.
Let $T'(u)$ be the number of vertices $v_i\in V(C_{\ell})$ such that $h_{i-1}=\{v_{i-1},v_{i},u\}$,  $v_{i+1}\in S(u)$ and $v_i$ is not adjacent to  $v_{i+2}$ with a hyperedge in $\h$.
By Claim~\ref{claim:+}, for every vertex $u\in V(\h')$, the number of adjacent vertices from $V(C_{\ell})$ is 
\[
\abs{S(u)}+\abs{L(u)}+\abs{R(u)}\leq \left( \frac{k-2}{2}-\abs{L(u)}+T(u)+T'(u)\right)+2\abs{L(u)}=\frac{k-2}{2}+\abs{L(u)}+T(u)+T'(u).
\]
Note that now we have number of edges in $\partial \h_{V(C_{\ell})}\leq \binom{k-2}{2}-\sum_{u\in V(\h')}T'(u)$ and $x\leq k-2-\sum_{u\in V(\h')}T(u)$. Thus we have $e_{\partial \h }(V(C_{\ell}),V(\h'))$ is at most
\begin{equation}\label{Equation:crossing}
   \sum_{u\in V(\h')} \left( \frac{k-2}{2}+\abs{L(u)}+T(u)+T'(u)\right)\leq \frac{(n-k+2)(k-2)}{2}+(k-2)+\sum_{u\in V(\h')} T'(u). 
\end{equation}
 
Finally we are done since $e(\partial \h)$ is at most
\begin{equation}\label{Equation_robust}
  \binom{k-2}{2}-\sum_{u\in V(\h')} T'(u)+\frac{(n-k+2)(k-2)}{2}+(k-2)+\sum_{u\in V(\h')} T'(u)+\frac{n-(k-2)}{2}\leq \frac{n(k-1)}{2}  
\end{equation}

For the remaining case, we need another claim, which is a strengthening of Claim~\ref{claim:+,++}.

\begin{claim}\label{claim:+,++,+++}
 Let $\{u_1,u_2,u_3\}$ and $\{u_1,u_4,u_5\}$ be  hyperedges of $\h'$ sharing a vertex. 
 Then for each $i$ and $j$ such that $2\leq i \leq 3<j\leq 5$ we have 
 $(S(u_i)\cup L(u_i))\cap (S(u_j)\cup L(u_j))^-=\emptyset$, 
 $(S(u_i)\cup L(u_i))\cap (S(u_j)\cup L(u_j))^{--}=\emptyset$, and 
 $(S(u_i)\cup L(u_i))\cap S(u_j)^{---}=\emptyset$. 
\end{claim}
We omit the proof of Claim~\ref{claim:+,++,+++} since it is the same as  the proofs of the Claim~\ref{claim:+} and Claim~\ref{claim:+,++}.

 If $\ell=k-3$ then $\h'$ contains no Berge path of length three. Therefore connected components of $\h'$ are edges of size two, hyperedges of size three, complete graphs on three vertices with edges of size two, and stars (hyperedges of size two and three all sharing the same vertex). 

As in the previous case Equation~\ref{Equation:crossing}, we have $e_{\partial \h }(V(C_{\ell}),V(\h'))$ is at most 
\begin{equation}\label{Equation:crossing_k-2}
 \sum_{u\in V(\h')} \left( \frac{k-3}{2}+\abs{L(u)}+T(u)+T'(u)\right)\leq \frac{(n-k+2)(k-3)}{2}+k-3+\sum_{u\in V(\h')} T'(u).   
\end{equation}

 Let $f_1,f_2,\dots,f_y\in \h'$ be the hyperedges of size three sharing a vertex $u'$ for some $y>1$.   
 Let us denote $Y:=\bigcup_{i=1}^{y} f_i\setminus \{u'\}$. 
 Let $z$ be the maximum of $\abs{(S(u)\cup L(u))}$ for $u\in Y$ achieved for a vertex $w$. 
 Then for the vertex $w'$ adjacent to $w$,  $w'\in Y$, we have $\abs{(S(w')\cup L(w'))}\leq z$ but for every other vertex $w''\in Y\setminus \{w,w'\}$ we have $\abs{(S(w'')\cup L(w''))}\leq k-3-3z$ by Claim~\ref{claim:+,++,+++}. Therefore we have 
 \[
 \sum_{y\in Y} \abs{(S(y)\cup L(y))}\leq \frac{k-3}{4}\abs{Y}. 
 \]
 Note that each vertex from $Y$ is incident to $\frac{3}{2}$ edges in $\partial \h'$ on average.
 The vertex set of $\h'$ is partitioned into two parts, the first set $U'$ contains vertices that are incident to hyperedges of size three and are contained in a path of length two in $\h$, and the second set $U''$ contains the rest of the vertices.
 For  vertices  $u\in U''$ we use bound from Equation~\ref{Equation:crossing_k-2} restricted to $U'$, $\abs{(S(u)\cup L(u))}\leq \frac{k-3}{2}$. Note that each such vertex is incident to at most one edge in $\partial \h'$ on average. Finally, we have 
 \[
e(\partial \h)\leq \binom{k-3}{2}+(k-3)+(n-(k-3)-\abs{U'})\left(\frac{k-3}{2}+1\right)+\abs{U'}(\frac{k-3}{4}+\frac{3}{2})<\frac{n(k-1)}{2},
\]
a contradiction.
And the proof of Theorem~\ref{thm:main} is complete.

\section*{Acknowledgements}
The research of the first author is partially supported by the National Research, Development and Innovation Office -- NKFIH, grant K116769, K132696 and SNN117879.  
The research of the second author was supported by the Institute for Basic Science (IBS-R029-C4).

\section{Concluding Remarks}

\bibliography{Proposal.bib}
\end{document}